\documentclass[12pt,
]{amsart}
\usepackage{%amsmath,
amssymb, graphicx
}
\usepackage{mathrsfs}
\usepackage{amsthm}
\usepackage{graphicx,color}

\newtheorem{theorem}{Theorem}
\newtheorem{lemma}{Lemma}
\newtheorem{proposition}{Proposition}

\newtheorem{remark}{Remark}
\theoremstyle{remark}

\date{\today}

\newcommand{\slu}[1]{\textcolor[rgb]{0.00,0.00,1.00}{#1}}

\title[Increasing stability of a linearized IBVP]{Increasing stability of a linearized inverse boundary value problem for a nonlinear Schr\"odinger equation on transversally anisotropic manifolds}
  \author[S. Lu]{Shuai Lu}
\address{School of Mathematical Sciences, SKLCAM and LMNS, Fudan University, Shanghai 200433, China}
\email{slu@fudan.edu.cn}
\thanks{S. Lu is supported by NSFC (No.11925104), Science and Technology Commission of Shanghai Municipality (21JC1400500). }
  \author[J. Zhai]{Jian Zhai}
  \address{Corresponding author. School of Mathematical Sciences, Fudan University, Shanghai 200433, China}
  \email{jianzhai@fudan.edu.cn}

\begin{document}
\begin{abstract}
We consider the problem of recovering a nonlinear potential function in a nonlinear Schr\"odinger equation on transversally anisotropic manifolds from the linearized Dirichlet-to-Neumann map at a large wavenumber. By calibrating the complex geometric optics (CGO) solutions according to the wavenumber, we prove the increasing stability of recovering the coefficient of a cubic term as the wavenumber becomes large.
\end{abstract}
\maketitle

\section{Introduction}
In this article we investigate the inverse boundary value problem for the Helmholtz equation with cubic nonlinearity
\begin{equation}\label{main_eq}
\Delta_g u+k^2u-cu^3=0,\quad\text{in }M,
\end{equation}
where $(M,g)$ is an $n$-dimensional compact oriented Riemannian manifold with smooth boundary $\partial M$ and $n\geq 3$. Here $\Delta_g$ is the Laplace-Beltrami operator, given in local coordinates by
\[
\Delta_g u=\det(g)^{-1/2}\partial_i(\det(g)^{1/2}g^{ij}\partial_ju),
\]
where the Riemannian metric $g=(g_{ij}(x))$ and $g^{-1}=(g^{ij}(x))$. We consider the inverse problems of recovering $c$ from boundary measurement, which is often modeled as the Dirichlet-to-Neumann (DtN) map.

 Much work has been done for the corresponding problems for the linear equation
\begin{equation}\label{linearequation}
\Delta_g u+k^2u-cu=0,\quad\text{in }M.
\end{equation}
When $k=0$, the problem is closely related to the Calder\'on's problem arising from Electrical Impedance Tomography (EIT). For the case where $g$ is Euclidean, the global uniqueness was first established in the seminal paper \cite{sylvester1987global}. Logarithmic stability of such problem was proved in \cite{alessandrini1988stable}, and found to be optimal in \cite{mandache2001exponential}. For the case $g$ is not Euclidean, the problem is in general still open. Much work has been done with further assumption that $(M,g)$ is conformally transversally anisotropic, i.e., $(M,g)\subset\subset(\mathbb{R}\times M_0,g)$, $g=d(e\oplus g_0)$, $(\mathbb{R},e)$ is the Euclidean line, $(M_0,g_0)$ is some compact $(n-1)$-dimensional manifold with boundary, and $d$ is some smooth positive function. With further assumption that  $(M_0,g_0)$ is simple, meaning $\partial M_0$ is convex, $M_0$ is simply connected and has no conjugate points, the uniqueness of $c$ was proved in \cite{dos2009limiting}. A double-logarithmic type stability estimate, under the same assumptions, was established in \cite{caro2014stability}. Under a more general assumption that the geodesic ray transform is injective on $(M_0,g_0)$, the uniqueness was proved in \cite{ferreira2016calderon}.

Since the work \cite{kurylev2018inverse}, much attention has been paid to inverse problems for nonlinear hyperbolic equations. For many nonlinear equations, the nonlinearity actually helps in solving the related inverse problems, whereas the problems for corresponding linear equations are still open. For recent works on inverse boundary value problems for nonlinear hyperbolic equations, we refer to \cite{chen2021detection,chen2021inverse,uhlmann2020determination,hintz2022dirichlet,uhlmann2022inverse,lassas2020uniqueness} and references therein. The inverse boundary value problem for the nonlinear elliptic equation
\begin{equation}\label{eq_no1}
\Delta_g u+cu^m=0,\quad\text{in }M
\end{equation}
$m\geq 3$,
 was considered in \cite{lassas2021inverse,feizmohammadi2020inverse}, where uniqueness was proved under very mild geometrical assumptions on $(M_0,g_0)$. For this problem, the nonlinearity was utilized in an essential way. In contrast with the study for the linear model \cite{dos2009limiting,caro2014stability,ferreira2016calderon} for which the problem is related to the geodesic ray transform on $(M_0,g_0)$, a pointwise recovery is possible with the nonlinear model \eqref{eq_no1}. The study on inverse boundary value problems for nonlinear elliptic equations goes back to \cite{isakov1995global,sun1996quasilinear,sun1997inverse}. For more recent works, we refer to \cite{assylbekov2021inverse,carstea2021calderon,krupchyk2020partial,kian2022partial,carstea2020inverse,feizmohammadi2021inverse} and the references therein.\\

In this article, we are mainly interested in the stability of the inverse boundary value problem for \eqref{main_eq}, especially its behavior when the wavenumber $k$ increases. Recently, it was observed that for many inverse problems for Helmholtz equations, the stability improves as the wavenumber grows \cite{isakov2007increased,isakov2011increasing,isakov2014increasing,BaoetalReview2015,isakov2016increasing,isakov2020linearized}. For the equation \eqref{linearequation} with $g$ to be Euclidean, the problem of recovering $c$ from linearized DtN map was considered in \cite{lu2022increasing, zlx2022submitted}. It was proved therein that the stability approaches to a H\"older type as $k$ goes to infinity. Numerical results also corroborate this behavior. When $g$ is Euclidean, similar increasing stability results have already been obtained in \cite{lu2022increasing, zlx2022submitted}. In this paper, we will consider the case where $g$ is transversally anisotropic with $(M_0,g_0)$ simple. The main technical difficulty is that the CGO solutions for the equation
\[
\Delta_gu+k^2u=0
\]
 used here are more complicated and can only be constructed asymptotically. The underlying reason is that a non-Euclidean metric $g$ causes geometrical scattering of waves. We emphasize here that, to the best of our knowledge, our result shall be the first increasing stability result on transversally anisotropic manifolds.\\

 We will consider recovering $c$ from two different linearized DtN maps, which are closely related. We will detail these two linearizations below.

 \noindent\textbf{Multilinearization.} The first linearization has been used in \cite{feizmohammadi2020inverse,lassas2021inverse}. Take $\epsilon=(\epsilon_1,\epsilon_2,\epsilon_3)$ where each small variable $\epsilon_j>0$, and consider the solution $u_\epsilon$ to \eqref{main_eq} with Dirichlet boundary value
 \[
 f_\epsilon=\epsilon_1f_1+\epsilon_2f_2+\epsilon_3f_3.
 \]
 It is clear that $w:=\partial_{\epsilon_1}\partial_{\epsilon_2}\partial_{\epsilon_3}u_\epsilon\big\vert_{\epsilon=0}$ solves the equation
 \begin{equation}\label{eq_w}
 \Delta_g w+k^2w=6cv_1v_2v_3,\quad w\vert_{\partial M}=0,
 \end{equation}
 where $v_j$, $j=1,2,3$, solves the equation
 \begin{equation}\label{eq_vj}
 \Delta_g v_j+k^2v_j=0,\quad v_j\vert_{\partial M}=f_j.
 \end{equation}
 Therefore we can define the multi-linear map
 \[
 D^3_0\Lambda_c:\left(C^2(\partial M)\right)^3\rightarrow L^{4/3}(\partial M)
 \]
 such that
 \[
 D^3_0\Lambda_c(f_1,f_2,f_3)=\partial_\nu w.
 \]
 Next, we verify the mapping properties mentioned above. Notice that for any $f_j\in C^2(\partial M)$, there exists a unique solution $v_j\in W^{2,4}(M)$ to \eqref{eq_vj} such that
 \[
 \|v_j\|_{W^{2,4}(M)}\leq \|f_j\|_{C^2(\partial M)}.
 \]
 Then $cv_1v_2v_3\in L^{4/3}(M)$ with the estimate
 \[
 \|cv_1v_2v_3\|_{L^{4/3}(M)}\leq C\|v_1\|_{L^4(M)}\|v_2\|_{L^4(M)}\|v_3\|_{L^4(M)}\leq C\|f_1\|_{C^2(\partial M)}\|f_2\|_{C^2(\partial M)}\|f_3\|_{C^2(\partial M)}.
 \]
According to the $L^p$ theory for elliptic equations, the solution $w$ to \eqref{eq_w} satisfies the estimate
 \[
 \|w\|_{W^{2,4/3}(M)}\leq C\|cv_1v_2v_3\|_{L^{4/3}(M)}\leq C\|f_1\|_{C^2(\partial M)}\|f_2\|_{C^2(\partial M)}\|f_3\|_{C^2(\partial M)},
 \]
 and thus
 \[
 \|\partial_vw\|_{L^{4/3}(\partial M)}\leq C\|cv_1v_2v_3\|_{L^{4/3}(M)}\leq C\|f_1\|_{C^2(\partial M)}\|f_2\|_{C^2(\partial M)}\|f_3\|_{C^2(\partial M)}.
 \]
 If we denote
 \[
 \epsilon=\sup_{\|f_j\|_{C^2(\partial M)}\leq 1}\|D^3_0\Lambda_c(f_1,f_2,f_3)\|_{L^{4/3}(\partial M)}=\sup_{\|f_j\|_{C^2(\partial M)}\leq 1}\|\partial_\nu w\|_{L^{4/3}(\partial M)},
 \]
 then for any $f_j\in C^2(\partial M)$, we have
 \[
 \|D^3_0\Lambda_c(f_1,f_2,f_3)\|_{L^{4/3}(\partial M)}\leq \epsilon \|f_1\|_{C^2(\partial M)}\|f_2\|_{C^2(\partial M)}\|f_3\|_{C^2(\partial M)}.
 \]

 \noindent\textbf{Linearization with respect to $c$.} This second linearization has been used in \cite{lu2022increasing}. Now, let us consider the (nonlinear) map $\Lambda'_c:f\mapsto \partial_\nu v_f\vert_{\partial M}$. Here $v_f$ solves the equation
 \[
 \begin{cases}
 (\Delta_g+k^2)v_f=cu_f^3\quad \text{in }M,\\
 v_f=0\quad \text{on }\partial M,
 \end{cases}
 \]
 where $u_f$ is the solution to
 \[
 \begin{cases}
 (\Delta_g+k^2)u_f=0\quad \text{in }M,\\
 u_f=f\quad \text{on }\partial M.
 \end{cases}
 \]
 For any $f\in C^2(\partial M)$, there is a unique solution $u_f\in W^{2,4}(M)$ such that
 \[
 \|u_f\|_{W^{2,4}(M)}\leq \|f\|_{C^2(\partial M)}.
 \]
 Then $cu_f^3\in L^{\frac{4}{3}}(M)$ with the estimate
 \[
 \|cu_f^3\|_{L^{\frac{4}{3}}(M)}\leq C\|u_f\|_{L^{4}(M)}^3\leq C\|f\|_{C^2(\partial M)}^3.
 \]
 Therefore, the solution $v_f$ satisfies the estimate
 \[
 \|v_f\|_{W^{2,\frac{4}{3}}(M)}\leq C\|f\|_{C^2(\partial M)}^3,
 \]
 and consequently
 \[
 \|\partial_\nu v_f\|_{L^{\frac{4}{3}}(\partial M)}\leq C\|f\|_{C^2(\partial M)}^3.
 \]
 Denote
 \[
 \epsilon=\sup_{\|f\|_{C^2(\partial M)}=1}\|\Lambda_c' f\|_{L^{\frac{4}{3}}(\partial M)}=\sup_{\|f\|_{C^2(\partial M)}=1}\|\partial_\nu v_f\|_{L^{\frac{4}{3}}(\partial M)}.
 \]
For any $f\in C^2(\partial M)$, we have
 \[
 \begin{cases}
 (\Delta_g+k^2)v_f=-c\|f\|_{C^2(\partial M)}^3u_{f/\|f\|_{C^2}(\partial M)}^3\quad \text{in }M,\\
 v_f=0\quad \text{on }\partial M.
 \end{cases}
 \]
 Therefore $v_f=\|f\|_{C^2(\partial M)}^3v_{f/\|f\|_{C^2(\partial M)}}$, and
 \[
 \|\Lambda_c' f\|_{L^{\frac{4}{3}}(\partial M)}\leq \epsilon\|f\|_{C^2(\partial M)}^3.
 \]

 We will consider the problem:\\

 \textbf{Recover $c(x)$ from the linearized DtN maps $D^3_0\Lambda_c$ or $\Lambda_c'$.}\\

 The rest of this article is organized as follows. In Section \ref{CGO}, we review the construction of CGO solutions and  sophisticatedly calibrate their behaviors with respect to several parameters. In Section \ref{proof}, we state and prove the main theorems.

\section{Complex Geometrical Optics solutions}\label{CGO}
In this section, we review the properties of CGO solutions on transversally anisotropic manifolds constructed in \cite{ferreira2016calderon}. The form of CGO solutions is slightly different since we have a nonzero wavenumber $k$ here. For our purposes, we need to keep track on how the solutions depend on the wavenumber $k$ and some parameters $\tau$ and $\lambda$ that will be introduced later.

Throughout this section, let $(M,g)\subset\subset (\mathbb{R}\times M_0,g)$ be a transversally anisotropic manifold with $g=e\oplus g_0$, where $e$ is the Euclidean metric on $\mathbb{R}$. Assume further that $(M_0,g_0)$ is an $(n-1)$-dimensional simple Riemannian manifold with smooth boundary $\partial M_0$. Recall that $(M_0,g_0)$ is called simple if  $\partial M_0$ is strictly convex and any two points $x,y\in M_0$ can be connected by a unique geodesic, contained in $M_0$, depending smoothly on $x$ and $y$. We will construct CGO solutions satisfying the equation
\begin{equation}
\Delta_g u+k^2u=0.
\end{equation}
Denote $x=(x_1,x')$ to be the coordinate system on $\mathbb{R}\times M_0$. Then we can write
\[
\Delta_g=\partial_1^2+\Delta_{g_0},
\]
where $\Delta_{g_0}$ is the Laplace-Beltrami operator on $(M_0,g_0)$.

 Let $\tau,\lambda$ be real numbers, where $|\tau|\geq 1$.
We will construct a family of solutions
\[
u_{\tau+\mathrm{i}\lambda}=e^{(\tau+\mathrm{i}\lambda)x_1}(\widetilde{v}_{\tau+\mathrm{i}\lambda}(x')+r(x)),
\]
 where $\widetilde{v}_{\tau+\mathrm{i}\lambda}$ is a family of functions on $M_0$, and the remainder term $r\rightarrow 0$ as $\tau\rightarrow+\infty$. We remark here that $\widetilde{v}_{\tau+\mathrm{i}\lambda}$ does not depend on $x_1$.

Notice that
\[
\begin{split}
&e^{-(\tau+\mathrm{i}\lambda)x_1}(-\Delta_g-k^2)e^{(\tau+\mathrm{i}\lambda)x_1} u\\
=&(-\partial_1^2+2(\tau+\mathrm{i}\lambda)\partial_1-\Delta_{g_0}-k^2-(\tau+\mathrm{i}\lambda)^2) u.
\end{split}
\]
From now on, we take $\tau,\lambda$ such that $k^2+\tau^2-\lambda^2\geq 1$.\\

 We first construct Gaussian beam quasimode $\widetilde{v}_{\tau+\mathrm{i}\lambda}(x')$ on $M_0$ such that $\widetilde{v}_{\tau+\mathrm{i}\lambda}(x')$ is concentrated near a geodesic in the high frequency limit.

 Assume that $\gamma:[0,L]\rightarrow M_0$ is a geodesic in $M_0$ such that $\dot{\gamma}(0)$ and $\dot{\gamma}(L)$ are nontangential vectors on $\partial M_0$ and $\gamma(t)\in M^{\mathrm{int}}$ for $0<t<L$.
Fix the values of $K$ and $m$, the solutions $\widetilde{v}_{\tau+\mathrm{i}\lambda}$ would be compactly supported in a neighborhood of $\gamma$ and satisfy the estimates
\begin{equation}\label{gaussianp1}
\|(-\Delta_{g_0}-k^2-(\tau+\mathrm{i}\lambda)^2)\widetilde{v}_{\tau+\mathrm{i}\lambda}\|_{H^m(M_0)}\leq C(k^2+\tau^2-\lambda^2)^{-\frac{K}{2}}e^{\sigma|\lambda|},
\end{equation}
and
\begin{equation}\label{gaussianp2}
\|\widetilde{v}_{\tau+\mathrm{i}\lambda}\|_{L^4(M_0)}\leq Ce^{\sigma|\lambda|},\quad \|\widetilde{v}_{\tau+\mathrm{i}\lambda}\|_{L^4(\partial M_0)}\leq Ce^{\sigma|\lambda|}
\end{equation}
for some positive constants $C$ and $\sigma$.

To simplify the notations, let $s=\sqrt{k^2+(\tau+\mathrm{i}\lambda)^2}$, and denote $\widetilde{v}_s=\widetilde{v}_{\tau+\mathrm{i}\lambda}$. Here and throughout the paper we consider $\sqrt{z}$ defined on $\{z:\Re z>0\}$ such that
$
\Re\sqrt{z}>0
$
everywhere.

 We will construct $\widetilde{v}_s$ such that
\[
(-\Delta_{g_0}-s^2)\widetilde{v}_s\sim 0.
\]
Note that $|s|\geq\sqrt{k^2+\tau^2-\lambda^2}$.
The Gaussian beam solutions are of the form
\[
\widetilde{v}_s=(\sqrt{k^2+\tau^2-\lambda^2})^{\frac{n-2}{8}}e^{\mathrm{i}s\Theta}a.
\]
The construction of above solutions will be carried out in the Fermi coordinates in a neighborhood of the geodesic $\gamma$ (cf. e.g. \cite[Lemma 3.5]{ferreira2016calderon}). Let us briefly recall that the Fermi coordinates can be constructed in the following way. First we choose $\{v_2,\cdots,v_{n-1}\}$ in $T_{\gamma(0)}M_0$ such that $\{v_1=\dot{\gamma}(0),v_2,\cdots,v_{n-1}\}$ is an orthonormal basis of $T_{\gamma(0)}M_0$. Let $E_{\alpha}(t)$ be the parallel transport of $v_\alpha$ along the geodesic $\gamma$. Then $\{\dot{\gamma}(t),E_2(t),\cdots,E_{n-1}(t)\}$ is an orthonormal basis of $T_{\gamma(t)}M_0$. Inverting the map
\[
F(t,y)=\exp_{\gamma(t)}\left(\sum_{\alpha=2}^{n-1}y^\alpha E_\alpha(t)\right)
\]
gives the Fermi coordinates $(t,y)$ near $\gamma(0,L)$ such that the geodesic $\gamma(t)$ corresponds to $\{y=0\}$. Here we have used the fact that $\gamma$ is not self-intersecting. Furthermore, under the Fermi coordinates the metric $g$ satisfies
\[
g^{jk}\vert_{\gamma(t)}=\delta^{jk},\quad\partial_ig^{jk}\vert_{\gamma(t)}=0.
\]

Let
\[
v_s=e^{\mathrm{i}s\Theta}a,
\]
where $\Theta$ and $a$ are smooth complex functions near $\gamma$ with $a$ supported in $\{|y|\leq\delta'/2\}$.
By calculation, one has
\[
(-\Delta_{g_0}-s^2)v_s=e^{\mathrm{i}s\Theta}\left(s^2[(\langle\mathrm{d}\Theta,\mathrm{d}\Theta \rangle_{g_0}-1)a]-\mathrm{i}s[2\langle\mathrm{d}\Theta,\mathrm{d}a\rangle_{g_0}+(\Delta_{g_0}\Theta)a]-\Delta_{g_0} a\right).
\]
We first choose $\Theta$ such that
\[
\langle\mathrm{d}\Theta,\mathrm{d}\Theta\rangle_{g_0}=1,\text{  up to $N$th order on $\gamma$}.
\]
This can be done by looking for $\Theta$ of the form $\Theta=\sum_{j=0}^N\Theta_j$ where $\Theta_j$ is a polynomial of degree $j$ in $y$. In particular, we can choose $\Theta_0(t,y)=t$, $\Theta_1(t,y)=0$. For the construction of $\Theta_2$, one can write
\[
\Theta_2(t,y)=\frac{1}{2} y\cdot H(t)y,
\]
where $H(t)$ is a smooth complex symmetric matrix solving some matrix Riccati equation (cf. \cite[pp.2595]{ferreira2016calderon}). Here $\cdot$ refers to the usual $\mathbb{R}^{n-2}$ inner product and $y\in\mathbb{R}^{n-2}$. In fact, one can choose $H(t)$ such that $\Im (H(t))$ is positive definite. This completes the construction of $\Theta_2$ and one can then successively construct $\Theta_3,\cdots, \Theta_N$ by solving additional ODEs. Then, the phase function $\Theta$ satisfies the following properties.
\[
\begin{split}
&\Theta(\gamma(t))=t,\quad \nabla \Theta(\gamma(t))=\dot{\gamma}(t),\\
&\Im(D^2\Theta(\gamma(t)))\geq 0, \quad \Im(D^2\Theta)(\gamma(t))(X,X)>0, \quad X\perp \dot{\gamma}(t),\, X\neq 0.
\end{split}
\]
The above properties of the phase function imply that the solution $v_s$ is exponentially decaying away from the geodesic $\gamma$.

Next we need to construct the amplitude $a$ such that
\[
-\mathrm{i}s[2\langle\mathrm{d}\Theta,\mathrm{d}a\rangle_{g_0}+(\Delta_{g_0}\Theta)a]-\Delta_{g_0} a=0\text{  up to $N$th order on $\gamma$}.
\]
We assume that $a$ is of the asymptotic form
\[
a=\chi(y/\delta')(a_0+s^{-1}a_1+s^{-2}a_1+\cdots+s^{-N}a_N)
\]
 with $a_0(\gamma(t))$ non-vanishing and $\chi$ is a smooth function with $\chi=1$ for $|y|\leq 1/4$ and $\chi=0$ for $|y|\geq 1/2$. Notice that the support of $a$ can be taken to be in an arbitrary neighborhood of $\gamma$ by choosing appropriate $\delta'$ small enough. Here $a_0,a_1,\cdots, a_N$ are independent of $s$ (thus independent of $\tau,\lambda,k$). It suffices to determine $a_j$ such that
 \[
 \begin{split}
 2\langle\mathrm{d}\Theta,\mathrm{d}a_0\rangle_{g_0} +(\Delta_{g_0}\Theta)a_0&=0\text{  up to $N$th order on $\gamma$},\\
  2\langle\mathrm{d}\Theta,\mathrm{d}a_j\rangle_{g_0} +(\Delta_{g_0}\Theta)a_j-\mathrm{i}\Delta_{g_0} a_{j-1}&=0\text{  up to $N$th order on $\gamma$ for }j=1,\cdots, N.
 \end{split}
 \]

We seek for $a_k$, $k=1,\cdots, N$, of the form
\[
a_k=\sum_{j=0}^N a_{k,j}(t,y),
\]
where $a_{k,j}$ is a complex homogeneous polynomial of order $j$ in $y$.
In particular $a_{0,0}$ satisfies the equation
\[
\partial_ta_{0,0}(t,0)+\frac{1}{2}\Delta_{g_0}\Theta(t,0)=0
\]
on the geodesic $\{y=0\}$.
Note that
\[
\Delta_{g_0}\Theta(t,0)=\mathrm{tr}(H(t)).
\]
We can take
\[
a_{0,0}(t,0)=c_0e^{-\frac{1}{2}\int_{t_0}^t\mathrm{tr}H(t')\mathrm{d}t'}.
\]
The details can be found in \cite[Proposition 3.1]{ferreira2016calderon}, and are omitted here.\\

To summarize, we have constructed a function $v_s=e^{\mathrm{i}s\Theta}a$ in a neighborhood of $\gamma$ where
\[
\begin{split}
\Theta(t,y)&=t+\frac{1}{2}y\cdot H(t)y+\widetilde{\Theta},\\
a(t,y)&=a_0+s^{-1}a_1+\cdots +s^{-N}a_N\chi(y/\delta'),\\
a_0(t,0)&=c_0e^{-\frac{1}{2}\int_0^t\mathrm{tr} H(t')\mathrm{d}t'}.
\end{split}
\]
Here $\widetilde{\Theta}=\mathcal{O}(|y|^3)$, and both $\Theta$ and each $a_j$ is independent of $s$.\\

Next, we derive a lower bound for $\Re s$ and an upper bound for $|\Im s|$.
\begin{lemma}\label{est_s}
Assume that $\tau^2+k^2-\lambda^2\geq 1$ and $\tau\geq 1$. For $s=\sqrt{k^2+(\tau+\mathrm{i}\lambda)^2}$, we have
\[
\Re s\geq c_0\sqrt{k^2+\tau^2-\lambda^2},
\]
\[
|\Im s|\leq \sqrt{5}|\lambda|.
\]
for some positive constant $c_0$.
\end{lemma}
\begin{proof}
For the first estimate, notice that
\[
\Re s=\Re\sqrt{k^2+\tau^2-\lambda^2+2\mathrm{i}\tau\lambda}=\sqrt{k^2+\tau^2-\lambda^2}\Re\sqrt{1+\frac{2\mathrm{i}\tau\lambda}{k^2+\tau^2-\lambda^2}}\geq c_0\sqrt{k^2+\tau^2-\lambda^2},
\]
where
\[
c_0=\inf_{t\in\mathbb{R}}\Re(\sqrt{1+\mathrm{i}t})>0.
\]

Now, note that
\[
\sqrt{k^2+\tau^2-\lambda^2+2\mathrm{i}\tau\lambda}-\sqrt{k^2+\tau^2}=\sqrt{k^2+\tau^2}\left(\sqrt{1+\frac{-\lambda^2+2\mathrm{i}\tau\lambda}{k^2+\tau^2}}-1\right).
\]
%for some $z_0$ between $0$ and $\frac{-\lambda^2+2\mathrm{i}\tau\lambda}{k^2+\tau^2}$.
%For $k^2+\tau^2-\lambda^2\geq 1$ and $\tau \geq 1$, we have
%\[
%\left|\frac{-\lambda^2+2\mathrm{i}\tau\lambda}{k^2+\tau^2}\right|\leq \sqrt{\frac{\lambda^4+4\tau^2\lambda^2}{(k^2+\tau^2)^2}}\leq\sqrt{\frac{(k^2+\tau^2)^2+4\tau^2(\tau^2+k^2)}{(k^2+\tau^2)^2}}\leq\sqrt{5}.
%\]
% Thus,
%\[
%|1+z_0|^{-1/2}\leq C_0
%\]
%for some (different) constant $C_0$.
Using the fact
\[
|\sqrt{1+z}-1|=\left|\frac{z}{\sqrt{1+z}+1}\right|\leq |z|
\]
 (notice that $\Re\sqrt{1+z}>0$), we obtain
\[
\left|\sqrt{k^2+\tau^2-\lambda^2+2\mathrm{i}\tau\lambda}-\sqrt{k^2+\tau^2}\right|\leq \frac{|-\lambda^2+2\mathrm{i}\tau\lambda|}{\sqrt{k^2+\tau^2}}.
\]
%Since
%\[
%\frac{\lambda^4}{k^2+\tau^2}\leq \lambda^2
%\]
%and
%\[
%\frac{\tau^2\lambda^2}{k^2+\tau^2}\leq \lambda^2
%\]
For $\tau^2+k^2\geq \lambda^2$ and $\tau\geq 1$, we have
\[
\frac{|-\lambda^2+2\mathrm{i}\tau\lambda|}{\sqrt{k^2+\tau^2}}= \sqrt{\frac{\lambda^4+4\tau^2\lambda^2}{k^2+\tau^2}}\leq \sqrt{\frac{\lambda^2(k^2+\tau^2)+4(\tau^2+k^2)\lambda^2}{k^2+\tau^2}}=\sqrt{5}|\lambda|.
\]
 In particular,
\[
|\Im s|\leq \left|\sqrt{k^2+\tau^2-\lambda^2+2\mathrm{i}\tau\lambda}-\sqrt{k^2+\tau^2}\right|\leq\sqrt{5}|\lambda|.
\]
\end{proof}

Now let
\[
\widetilde{v}_{\tau+\mathrm{i}\lambda}=\widetilde{v}_s=(\sqrt{k^2+\tau^2-\lambda})^{\frac{n-2}{8}}v_s,
\]
and
\[
f=(-\Delta_{g_0}-s^2)\widetilde{v}_{\tau+\mathrm{i}\lambda}.
\]
By \cite[Proposition 3.1]{ferreira2016calderon}, we have that $f$ is of the form
\[
\begin{split}
f=&(\sqrt{k^2+\tau^2-\lambda^2})^{\frac{n-2}{8}}e^{\mathrm{i}s\Theta}(s^2h_2a_0+sh_1+\cdots s^{-(N-1)})h_{-(N-1)}-s^{-N}\Delta a_{-N})\chi(y/\delta')\\
&+(\sqrt{k^2+\tau^2-\lambda^2})^{\frac{n-2}{8}}e^{\mathrm{i}s\Theta}sb\widetilde{\chi}(y/\delta')
\end{split}
\]
where $h_j=0$ up to order $N$ on $\gamma$, $b$ vanishes near $\gamma$, and $\widetilde{\chi}$ is a smooth function with $\tilde{\chi}=0$ for $|y|\geq 1/2$.

Now notice that, using Lemma \ref{est_s},
\begin{equation}\label{normPhi}
\vert e^{\mathrm{i}s\Theta}\vert\leq e^{-(\Im s)(\Re\Theta)}e^{-(\Re s)(\Im\Theta)}\leq e^{\sigma|\lambda|}e^{-c\sqrt{k^2+\tau^2-\lambda^2}|y|^2}
\end{equation}
for some constant $\sigma>0$.
 Therefore, if we take $\delta'$ small enough,
\begin{equation}\label{asymp_f}
|f|\lesssim (\sqrt{k^2+\tau^2-\lambda^2})^{\frac{n-2}{8}}e^{\sigma\lambda}e^{-c\sqrt{k^2+\tau^2-\lambda^2}|y|^2}(|s|^2|y|^{N+1}+|s|^{-N}+|s|\mathcal{O}(|y|^{\infty}))
\end{equation}
in a neighborhood of $\gamma$.

\begin{remark}
	In above, we can take $\sigma=\sqrt{5}\sup|\Re\Theta|=\sqrt{5}\mathrm{diam}_{g_0}(M_0)$, where $\mathrm{diam}_{g_0}(M_0)$ is the diameter of $M_0$ w.r.t. the metric $g_0$ \textnormal{(}i.e., the supreme of lengths of all the geodesics in $(M_0,g_0)$\textnormal{)}.
\end{remark}
Also, we have
\[
|\widetilde{v}_s|\lesssim (\sqrt{\tau^2+k^2-\lambda^2})^{\frac{n-2}{8}} e^{\sigma|\lambda|}e^{-c\sqrt{\tau^2+k^2-\lambda^2}|y|^2}\chi(y/\delta').
\]
Then, we obtain
\[
\begin{split}
\int_{M_0}|\widetilde{v}_s|^4\mathrm{d}V_{g_0}\lesssim & (\sqrt{\tau^2+k^2-\lambda^2})^{\frac{n-2}{2}}e^{4\sigma|\lambda|}\int_0^\infty r^{n-3} e^{-c\sqrt{\tau^2+k^2-\lambda^2}r^2}\mathrm{d}r\\
\lesssim &e^{4\sigma|\lambda|}\int_0^\infty\rho^{n-3}e^{-c\rho^2}\mathrm{d}\rho\\
\lesssim &e^{4\sigma|\lambda|},
\end{split}
\]
and consequently
\[
\|\widetilde{v}_s\|_{L^4(M_0)}\lesssim  e^{\sigma|\lambda|}.
\]
%Similarly, we have
%\[
%\|\widetilde{v}_s\|_{L^4(\partial M_0)}\lesssim  e^{\sigma\lambda}.
%\]

By taking $N$ large enough, we have
\[
\begin{split}
\|f\|_{L^2(M_0)}=&\|(-\Delta_{g_0}-s^2)\widetilde{v}_s\|_{L^2(M_0)}\\
 \lesssim&\|(\sqrt{\tau^2+k^2-\lambda^2})^{\frac{n-2}{8}}e^{\sigma|\lambda|}e^{-c\sqrt{\tau^2+k^2-\lambda^2}|y|^2}(|s|^2|y|^{N+1}+|s|^{-N})\|_{L^2( M_0)}\\
 \lesssim& (\sqrt{\tau^2+k^2-\lambda^2})^{-R}e^{\sigma|\lambda|}.
\end{split}
\]
Here and throughout the paper, $\sigma$ is a general constant, which may vary from step to step.
In above we have used the fact that
\[
\begin{split}
|s|^4e^{\sigma\lambda}\leq & (k^2+\tau^2-\lambda^2)^2e^{\sigma|\lambda|}+\tau^2\lambda^2e^{\sigma|\lambda|}\\
\leq &(k^2+\tau^2-\lambda^2)^2e^{\sigma|\lambda|} +(\tau^2+\lambda^2+k^2)^2e^{\sigma|\lambda|}\\
\lesssim& (k^2+\tau^2-\lambda^2)^2e^{\sigma|\lambda|}+\lambda^4e^{\sigma|\lambda|}\\
\lesssim &(k^2+\tau^2-\lambda^2)^2e^{\sigma'|\lambda|}
\end{split}
\]
with $\sigma'>\sigma$ for $k^2+\tau^2-\lambda^2\geq 1$. For taking $k$ derivatives of $f$ brings at most $k$ powers of $s$ to the front in \eqref{asymp_f}, thus, we can take $N=N(K,m)$ large enough such that
\[
\|f\|_{H^m(M_0)}=\|(-\Delta_{g_0}-s^2)\widetilde{v}_s\|_{H^m(M_0)}=\mathcal{O}((k^2+\tau^2-\lambda^2)^{-K/2}e^{\sigma'|\lambda|}).
\]
~\\

Next, we construct the remainder term $r$ for the CGO solutions.
Let $r=r(x_1,x')$ solve the equation
\[
(-\partial_1^2+2(\tau+\mathrm{i}\lambda)\partial_1-\Delta_{g_0}-k^2-(\tau+\mathrm{i}\lambda)^2)r(x_1,x')=(\Delta_{g_0}+k^2+(\tau+\mathrm{i}\lambda)^2)\widetilde{v}_{\tau+\mathrm{i}\lambda},
\]
which can be rewritten as
\[
(-\partial_1^2+2(\tau+\mathrm{i}\lambda)\partial_1-\Delta_{g_0}-s^2)r=(\Delta_{g_0}+s^2)\widetilde{v}_{\tau+\mathrm{i}\lambda}.
\]
The above equation is solvable by the following lemma, and the solution $r$ satisfies the estimate
\[
\|r\|_{H^m(M)}\leq C\|(\Delta_{g_0}+s^2)\widetilde{v}_{\tau+\mathrm{i}\lambda}\|_{H^m(M_0)}=C\|f\|_{H^m(M_0)}\leq C(k^2+\tau^2-\lambda^2)^{-K/2}e^{\sigma|\lambda|}.
\]
\begin{lemma}\label{lemma_cylinder}
Consider the equation
\begin{equation}\label{eq_uf}
(-\partial_1^2+2(\tau+\mathrm{i}\lambda)\partial_1-\Delta_{g_0}-k^2-(\tau+\mathrm{i}\lambda)^2)r=f,\quad \text{in }\mathcal{N}=(0,T)\times M_0.
\end{equation}
There is a constant $C$ such that for any $|\tau|\geq 1$, the above equation has a solution $r$ satisfying
\[
\|r\|_{H^{m}(\mathcal{N})}\leq \frac{C}{\tau}\|f\|_{H^m(\mathcal{N})}.
\]
If $\tau^2$ is not a Dirichlet eigenvalue of $-\Delta_{g_0}-k^2$ on $(M_0,g_0)$, then the solution is unique.
\end{lemma}
For the proof of Lemma \ref{lemma_cylinder}, we need the following lemma. Denote $L^2_\delta(\mathbb{R})$ be the space defined via the norm $\|f\|_{L^2_\delta}:=\|\langle x\rangle^\delta f\|_{L^2(\mathbb{R})}$.

\begin{lemma}\label{lemma_ode}
Let $z$ be a complex number with $\Re z\neq 0$, and consider the equation
\begin{equation}\label{ODE_Sa}
u'-z u=f\quad \text{in }\mathbb{R}.
\end{equation}
For any $f\in\mathscr{S}'(\mathbb{R})$ there is a unique solution $u\in\mathscr{S}'(\mathbb{R})$. Writing $S_z f:=u$, we have the mapping properties
\[
S_z:L^2_\delta(\mathbb{R})\rightarrow L^2_\delta(\mathbb{R})\quad \text{for all }\delta\in\mathbb{R},
\]
and
\[
\begin{split}
&\|S_z f\|_{L^2_\delta}\leq \frac{C_\delta}{|\Re z|} \|f\|_{L^2_\delta},\quad \text{if }|\Re z|\geq 1\\
&\|S_z f\|_{L^2_{-\delta}}\leq C_\delta  \|f\|_{L^2_\delta},\quad\text{if }\Re z\neq 0\text{ and }\delta>1/2.
\end{split}
\]
\end{lemma}
\begin{proof}
The proof is a minor modification of that of \cite[Proposition 4.4]{salo2013calderon}.

Taking the Fourier transforms of both sides of \eqref{ODE_Sa}, we have
\[
(\mathrm{i}\xi-\Re z-\mathrm{i}\Im z)\hat{u}=\hat{f}.
\]
Thus
\[
\hat{u}(\xi)=\mathscr{F}^{-1}\left\{\frac{1}{\mathrm{i}\xi-\Re z-\mathrm{i}\Im z}\hat{f}(\xi)\right\}.
\]
Denote $p(\xi)=\frac{1}{\mathrm{i}\xi-\Re z-\mathrm{i}\Im z}$, one can verify that
\[
\vert p^{\slu{(\ell)}}(\xi)\vert\leq \slu{\ell}!|\Re z|^{-\slu{\ell}-1},\quad \slu{\ell}=0,1,2,\cdots.
\]
Therefore $S_z:\mathscr{S}'(\mathbb{R})\rightarrow \mathscr{S}'(\mathbb{R})$, where $S_z f=\mathscr{F}^{-1}(p(\xi)\hat{f}(\xi))$, is continuous.

Let $f\in L^2_\delta(\mathbb{R})$, $\delta\in\mathbb{R}$. Notice that
\[
\|\hat{v}\|_{H^\delta}=\|\langle\,\cdot\,\rangle^\delta v\|_{L^2(\mathbb{R})}=\|v\|_{L^2_\delta}.
\]
Here $H^{\delta}=H^{\delta}(\mathbb{R})$ is the usual Sobolev space.
Fix $\slu{\ell}\geq |\delta|$,  for $|\Re z|\geq 1$, we have
\[
\|S_z f\|_{L^2_\delta}=\|p\hat{f}\|_{H^\delta}\leq C_\delta\|p\|_{W^{\slu{\ell},\infty}(\mathbb{R})}\|\hat{f}\|_{H^\delta}=C_\delta|\Re z|^{-1}\|f\|_{L^2_\delta}.
\]

Let $f\in L^1(\mathbb{R})$, and, without loss of generality, let $\Re z>0$. Applying the method of integrating factors to the ODE \eqref{ODE_Sa}, we have
\[
(ue^{-z t})'=fe^{-z t}.
\]
Therefore we have
\[
u(x)=-\int_{x}^\infty f(t)e^{-z(t-x)}\mathrm{d}t.
\]
Since $\Re z>0$, we have $|e^{-z(t-x)}|\leq 1$ for $t\geq x.$ Therefore
\[
\|u\|_{L^\infty(\mathbb{R})}\leq \|f\|_{L^1(\mathbb{R})}.
\]

Now let $f\in L_\delta^2(\mathbb{R})$ with $\delta>1/2$. We have
\[
\begin{split}
\|S_z f\|_{L^2_{-\delta}}=&\left(\int\langle t\rangle^{-2\delta}|S_z f(t)|^2\mathrm{d}t\right)^{1/2}\\
\leq &\left(\int\langle t\rangle^{-2\delta}\mathrm{d}t\right)^{1/2}\|S_z f\|_{L^\infty(\mathbb{R})}\\
\leq &c_\delta\|f\|_{L^1(\mathbb{R})}\\
=&c_\delta\int \langle t\rangle^{-\delta}\langle t\rangle^{\delta}|f(t)|\mathrm{d}t\\
\leq &c_\delta^2\|f\|_{L^2_\delta}.
\end{split}
\]
\end{proof}
%Fix $m$, for any $R\in\mathbb{N}$, we can take $K$ large enough such that
%\[
%\|r\|_{H^m(M)}\leq C(k^2+\tau^2-\lambda^2)^{-R/2}.
%\]
\begin{proof}[Proof of Lemma \ref{lemma_cylinder}]
Let $0<\omega_1^2\leq \omega_2^2\leq\cdots$ be the eigenvalues of the Dirichlet Laplacian $-\Delta_{g_0}$, and $\phi_j$ be the eigenfunctions such that
\[
-\Delta_{g_0}\phi_j=\omega^2_j\phi_j,\quad \phi_j\in H^1_0(M_0).
\]
Let
\[
\hat{f}(x_1,j)=\int_{M_0}f(x_1,x')\phi(x')\mathrm{d}V_{g_0}(x'),\quad j=1,2,\cdots,
\]
be the Fourier coefficients of $f(x_1,\cdot)$. One has the expansion
\[
f(x_1,x')=\sum_{j=1}^\infty \hat{f}(x_1,j)\phi_j(x').
\]
Assume that $r$ has the expansion
\[
r(x_1,x')=\sum_{j=1}^\infty \hat{r}(x_1,j)\phi_j(x').
\]
Insert the eigenfunction expansions of $f$ and $r$ into equation \eqref{eq_uf}, we obtain the ODEs
\begin{equation}\label{ODE}
\left(-\partial^2_1+2(\tau+\mathrm{i}\lambda)\partial_1+\omega_j^2-k^2-(\tau+\mathrm{i}\lambda)^2\right)\hat{r}(x_1,j)=\hat{f}(x_1,j).
\end{equation}
The symbol of the above ODE operator is $p(\xi_1,j)=\xi_1^2+2\mathrm{i}(\tau+\mathrm{i}\lambda)\xi_1-k^2+\omega^2-(\tau+\mathrm{i}\lambda)^2$. Notice that
\[
\Re p(\xi_1,j)=\xi_1^2-2\lambda\xi_1-k^2+\omega^2_j-\tau^2+\lambda^2,\quad \Im p(\xi_1,j)=2\tau(\xi_1-\lambda).
\]
Thus the symbol $p(\xi_1,j)$ is vanishing for $\tau\geq 1$ only if $\xi_1=\lambda$ and $-k^2+\omega_j^2-\tau^2=0$, which is impossible when $\tau^2$ is not an eigenvalue of $-\Delta_{g_0}-k^2$. This proves the uniqueness of the solution to \eqref{eq_uf}.

To show the existence, we observe that
\[
\begin{split}
&-\partial^2_1+2(\tau+\mathrm{i}\lambda)\partial_1+\omega_j^2-k^2-(\tau+\mathrm{i}\lambda)^2\\
=&-\left(\partial_1-\tau-\mathrm{i}\lambda+\sqrt{\omega^2_j-k^2}\right)\left(\partial_1-\tau-\mathrm{i}\lambda-\sqrt{\omega^2_j-k^2}\right).
\end{split}
\]
Using Lemma \ref{lemma_ode}, we have
\[
\hat{r}(\cdot,j)=S_{\tau+\mathrm{i}\lambda-\sqrt{\omega^2_j-k^2}}S_{\tau+\mathrm{i}\lambda+\sqrt{\omega^2_j-k^2}}\hat{f}(\cdot,j).
\]
Assume $\tau>0$ (the case $\tau<0$ is analogous). Then $\Re (\tau+\mathrm{i}\lambda+\sqrt{\omega^2_j-k^2})\geq 1$, and $\Re (\tau+\mathrm{i}\lambda-\sqrt{\omega^2_j-k^2})\neq 0$, so we have
\[
\|\hat{r}(\cdot,j)\|_{L_{-\delta}^2}\leq \frac{C}{\tau}\|\hat{f}(\cdot,j)\|_{L_\delta^2},
\]
and consequently
\[
\|r\|_{L^2(\mathcal{N})}\leq \frac{C}{\tau}\|f\|_{L^2(\mathcal{N})}.
\]
By taking $\partial_1^s$ of both sides of \eqref{ODE}, we obtain
\[
\|\partial^s_1\hat{r}(\cdot,j)\|_{L_{-\delta}^2}\leq \frac{C}{\tau}\|\partial^s_1\hat{f}(\cdot,j)\|_{L_\delta^2}.
\]
Multiplying both sides of \eqref{ODE} by $\omega_j$, we obtain
\[
\|\omega_j^s\hat{r}(\cdot,j)\|_{L_{-\delta}^2}\leq \frac{C}{\tau}\|\omega_j^s\hat{f}(\cdot,j)\|_{L_{\delta}^2}.
\]
Note that
\[
\|r\|_{H^m(\mathcal{N})}^2=\sum_{j=1}^\infty\sum_{s=0}^m \int |\partial_1^s\hat{r}(x_1,j)|^2+|\omega_j^s\hat{r}(x_1,j)|^2\mathrm{d}x_1.
\]
Therefore, we can conclude that
\[
\|r\|_{H^m(\mathcal{N})}\leq \frac{C}{\tau}\|f\|_{H^m(\mathcal{N})}.
\]
\end{proof}

We summarize the properties of CGO solutions in the following proposition.
\begin{proposition}
	Let $(M,g)$ be a transversally anisotropic manifold compactly supported in $(\mathbb{R}\times M_0,e\oplus g_0)$ with $(M_0,g_0)$ a simple manifold. Let $K,m\in \mathbb{N}$. For any $k,\tau,\lambda$, $k^2+\tau^2-\lambda^2\geq 1$, $\tau>1$, there is a solution of the equation $-\Delta_gu-k^2u=0$ having the form
	\begin{equation}
    u_{\tau+\mathrm{i}\lambda}=e^{(\tau+\mathrm{i}\lambda)x_1}(\widetilde{v}_{\tau+\mathrm{i}\lambda}+r),
	\end{equation}
	where $x_1$ is the coordinate along $\mathbb{R}$, $\widetilde{v}_{\tau+\mathrm{i}\lambda}$ is a family in $(M_0,g_0)$ satisfying \eqref{gaussianp1}, \eqref{gaussianp2} with $K$ chosen large enough, and the remainder term $r$ satisfies
	\[
	\|r\|_{H^m(M)}\lesssim (k^2+\tau^2-\lambda^2)^{-K/2}e^{\sigma|\lambda|}.
	\]
\end{proposition}
\section{Proof of the main result}\label{proof}
In this section we will provide the stability estimates for the problems under consideration, that is, recovering $c$ from the linearized DtN map $D^3_0\Lambda_c$ or $\Lambda_c'$. Since the inverse problem is linear, we only need to control the size of $c$ by its corresponding boundary data.\\

We first state a relation between the two linearizations $D^3_0\Lambda_c$ and $\Lambda_c'$. Notice that
\[
6v_1v_2v_3=(v_1+v_2+v_3)^3-(v_1+v_2)^3-(v_1+v_3)^3-(v_2+v_3)^3+v_1^3+v_2^3+v_3^3.
\]
Thus we obtain
\begin{equation}\label{relation}
\begin{split}
&D^3_0\Lambda_c(f_1,f_2,f_3)\\
=&\Lambda'_c(f_1+f_2+f_3)-\Lambda'_c(f_1+f_2)-\Lambda'_c(f_1+f_3)-\Lambda'_c(f_2+f_3)\\
&+\Lambda'_c(f_1)+\Lambda'_c(f_2)+\Lambda'_c(f_3).
\end{split}
\end{equation}

Because the operators $D^3_0\Lambda_c$ and $\Lambda'_c$ are related in the above way, we only need to consider the problem of recovering $c$ from $D^3_0\Lambda_c$.

The main result of this article is the following stability estimate.
\begin{theorem}\label{maintheorem}
	Denote $\epsilon=\sup_{\|f_j\|_{C^2(\partial M)}\leq 1}\|D^3_0\Lambda_c(f_1,f_2,f_3)\|_{L^{4/3}(\partial M)}$.
Let $\|c\|_{H^1}\leq \mathcal{M}$ and $k>1$, $\epsilon<1$, then	we have the stability estimate
	\begin{equation}\label{mainestimate}
	\|c\|_{L^2}^2\lesssim k^{12+\frac{3n}{4}}(\log k)\epsilon^2+E^{12+\frac{3n}{4}}(\log E)\epsilon+\frac{\log E}{(E^2+k^2)^{1/4}}+\frac{1}{(\log E)^2+(\log k)^2}.
	\end{equation}
	with $E=-\ln \epsilon$.
\end{theorem}
\begin{remark}
	We note that when $k\rightarrow+\infty$, the stability estimate \eqref{mainestimate} approaches to a H\"older type. When $k$ is small, the stability is a double logarithmic type.
\end{remark}

For the inverse problem of recovering $c$ from $\Lambda'_c$, we can obtain the same stability estimate simply by invoking \eqref{relation}.
\begin{theorem}
	Denote $\epsilon=\sup_{\|f\|_{C^2(\partial M)}\leq 1}\|\Lambda'_cf\|_{L^{4/3}(\partial M)}$.
	Let $\|c\|_{H^1}\leq \mathcal{M}$ and $k>1$, $\epsilon<1$, then	we have the stability estimate
	\begin{equation}
	\|c\|_{L^2}^2\lesssim k^{12+\frac{3n}{4}}(\log k)\epsilon^2+E^{12+\frac{3n}{4}}(\log E)\epsilon+\frac{\log E}{(E^2+k^2)^{1/4}}+\frac{1}{(\log E)^2+(\log k)^2}
	\end{equation}
	with $E=-\ln \epsilon$.
\end{theorem}
To prove the main theorem, we start with the Calder\'on-type identity
\begin{equation}
6\int_{M}c(x)v_1v_2v_3v_4\mathrm{d}V_g=\int_{\partial M}D^3_0\Lambda_c(f_1,f_2,f_3)f_4\mathrm{d}V_g,
\end{equation}
where each $v_j$ solves $\Delta_g v_j+k^2v_j=0$ with $v_j\vert_{\partial M}=f_j$. This can be obtained by integrating the equation \eqref{eq_w} against $v_4$. Therefore,
\begin{equation}\label{ineq_main}
\left|\int_{M}c(x)v_1v_2v_3v_4\mathrm{d}V_g\right|\leq C\epsilon (\Pi_{j=1}^3\|f_j\|_{C^2(\partial M)})\|f_4\|_{L^4(\partial M)}.
\end{equation}
For the proof of the stability estimate \eqref{mainestimate}, we need to plug the CGO solutions constructed in Section \ref{CGO} into the above inequality.

\subsection{Construction of four CGO solutions}
In this section, we will use similar settings as in \cite{lassas2021inverse}.
Let $\gamma,\eta\subset M_0$ be two geodesics intersecting only at one point $y_0\in M_0$, which is possible because $(M_0,g_0)$ is simple.

The CGO solutions are of the following forms
\[
\begin{split}
v_1(x)=e^{-(\tau+\mathrm{i}\lambda)x_1}(\widetilde{v}_{\tau+\mathrm{i}\lambda}(x')+r_1(x)),\quad v_2(x)=e^{(\tau-\mathrm{i}\lambda)x_1}(\overline{\widetilde{v}_{\tau+\mathrm{i}\lambda}(x')}+r_2(x)),\\
v_3(x)=e^{-\tau x_1}(\widetilde{w}_{\tau}(x')+r_3(x)),\quad v_4(x)=e^{\tau x_1}(\overline{\widetilde{w}_{\tau}(x')}+r_4(x)),
\end{split}
\]
where $\widetilde{v}_{\tau+\mathrm{i}\lambda}$ and $\widetilde{w}_{\tau}$ are the Gaussian beam solutions concentrating near $\gamma$ and $\eta$, respectively. Assume that
\[
\begin{split}
\widetilde{v}_{\tau+\mathrm{i}\lambda}=&(\sqrt{\tau^2+k^2-\lambda^2})^{\frac{n-2}{8}}e^{\mathrm{i}s\Phi(x')}\chi_1(x')a(x'),\\ \widetilde{w}_{\tau+\mathrm{i}\lambda}=&(\sqrt{\tau^2+k^2-\lambda^2})^{\frac{n-2}{8}}e^{\mathrm{i}\sqrt{k^2+\tau^2}\Psi(x')}\chi_2(x')b(x'),
\end{split}
\]
where $\chi_1$ is a cut-off function supported in a neighborhood of $\gamma$, and $\chi_2$ is a cut-off function supported in a neighborhood of $\eta$. The phase functions $\Phi$ and $\Psi$ satisfy
\[
\begin{split}
\Phi(\gamma(t))=t,\quad \nabla\Phi(\gamma(t))=\dot{\gamma}(t),\quad \Im(\nabla^2\Phi(\gamma(t)))\leq 0,\quad \Im(\nabla^2\Phi)\vert_{\dot{\gamma}(t)^\perp}>0,\\
\Psi(\eta(t))=t,\quad \nabla\Psi(\eta(t))=\dot{\eta}(t),\quad \Im(\nabla^2\Psi(\eta(t)))\leq 0,\quad \Im(\nabla^2\Psi)\vert_{\dot{\eta}(t)^\perp}>0.
\end{split}
\]
The amplitudes $a$ and $b$ admit the following asymptotics
\[
a=a_0+s^{-1}a_1+\cdots+s^{-N}a_N,\quad b=b_0+(\sqrt{k^2+\tau^2})^{-1}b_1+\cdots+(\sqrt{k^2+\tau^2})^{-N}b_N.
\]

Recall that one can choose $N$ large enough such that
\[
(\Delta_{g_0}+k^2+(\tau+\mathrm{i}\lambda)^2)\widetilde{v}_{\tau+\mathrm{i}\lambda}=\mathcal{O}((\tau^2+k^2-\lambda^2)^{-K/2}e^{\sigma|\lambda|}),
\]
\[
(\Delta_{g_0}+k^2+\tau^2)\widetilde{w}_{\tau}=\mathcal{O}((\tau^2+k^2-\lambda^2)^{-K/2}).
\]
Thus, the remainder term can be constructed to satisfy the estimates
\[
\begin{split}
\|r_j\|_{H^m(M)}\leq C(k^2+\tau^2-\lambda^2)^{-K/2}e^{\sigma|\lambda|},\\
\end{split}
\]
for $j=1,2$, and
\[
\begin{split}
\|r_j\|_{H^m(M)}\leq C(k^2+\tau^2-\lambda^2)^{-K/2},\\
\end{split}
\]
for $j=3,4$.
One can choose the value of $m$ such that
\[
\|r_j\|_{L^4(\partial M)}\leq \|r_j\|_{H^m(M)}\leq C(k^2+\tau^2-\lambda^2)^{-K/2}e^{\sigma|\lambda|}.
\]
for any $j=1,2$, and
\[
\begin{split}
\|r_j\|_{L^4(\partial M)}\leq \|r_j\|_{H^m(M)}\leq C(k^2+\tau^2-\lambda^2)^{-K/2},\\
\end{split}
\]
for $j=3,4$.

Notice that
\[
\begin{split}
\|e^{-(\tau+\mathrm{i}\lambda)x_1}\widetilde{v}_{\tau+\mathrm{i}\lambda}\|_{C^2(M)}
\leq &C(\sqrt{k^2+\tau^2-\lambda^2})^{\frac{n-2}{8}}(k^2+\tau^2+\lambda^2)|e^{-(\tau+\mathrm{i}\lambda)x_1}||e^{\mathrm{i}s\Phi(x')}|\\
%\leq & C(k^2+\tau^2+\lambda^2)^{\frac{n+14}{16}}e^{D\tau}e^{-(\Im s)(\Re\Phi)}e^{-(\Re s) (\Im\Phi)}\\
\leq &C (k^2+\tau^2+\lambda^2)^{\frac{n+14}{16}}e^{D\tau}e^{\sigma|\lambda|}
\end{split}
\]
for some constant $\sigma>0$, where we have used \eqref{normPhi}.
Using also $H^m(M)\subset\subset C^2(M)$ for $m$ sufficiently large, we obtain
\[
\begin{split}
\|v_j\|_{C^2(M)}\leq&\|e^{-(\tau+\mathrm{i}\lambda)x_1}\widetilde{v}_{\tau+\mathrm{i}\lambda}\|_{C^2(M)}+\|e^{-(\tau+\mathrm{i}\lambda)x_1}r_j\|_{C^2(M)}\\
\leq& C(k^2+\tau^2+\lambda^2)^{\frac{n+14}{16}}e^{\sigma|\lambda|}e^{D\tau},
\end{split}
\]
for $j=1,2$. Similarly,
\[
\begin{split}
	\|v_j\|_{C^2(M)}
	\leq C(k^2+\tau^2)^{\frac{n+14}{16}}e^{D\tau},
\end{split}
\]
for $j=3,4$.

Therefore, by \eqref{ineq_main} we have
\begin{equation}
\left|\int_{M}c(x)v_1v_2v_3v_4\mathrm{d}V_g\right|\lesssim \epsilon (k^2+\tau^2+\lambda^2)^{\frac{3n+42}{16}}e^{\sigma|\lambda|}e^{D\tau},
\end{equation}
with possibly different $\sigma$ and $D$.

\subsection{Proof of Theorem \ref{maintheorem}}
Denote
\[
\hat{c}(\lambda,\cdot)=\int e^{-\mathrm{i}\lambda x_1}c(x_1,\cdot)\mathrm{d}x_1
\]
to be the Fourier transform of $c$ in $x_1$.
Notice that, since $|\widetilde{v}_{\tau+\mathrm{i}\lambda}|^2|\widetilde{w}_{\tau}|^2$ is supported in a neighborhood of $y_0$,
\[
\begin{split}
\int_{M}c v_1v_2v_3v_4\mathrm{d}V_g=&\int_{-\infty}^\infty\int_{M_0} ce^{-2\mathrm{i}\lambda x_1}|\widetilde{v}_{\tau+\mathrm{i}\lambda}|^2|\widetilde{w}_{\tau}|^2\mathrm{d}V_{g_0}\mathrm{d}x_1+\mathcal{O}((k^2+\tau^2-\lambda^2)^{-K/2}e^{\sigma|\lambda|})\\
=&\int_{B_\delta(y_0)}\hat{c}(2\lambda,\cdot)e^{-2\sqrt{\tau^2+k^2}\Im(\Phi+\Psi)}A\mathrm{d}V_{g_0}+\mathcal{O}((k^2+\tau^2-\lambda^2)^{-K/2}e^{\sigma|\lambda|}).
\end{split}
\]
Here $\hat{c}(2\lambda,\cdot)$ is the Fourier transform of $c(x_1,x')$ in the $x_1$-variable, i.e.,
\[
\hat{c}(\lambda,x')=\int_{-\infty}^\infty e^{-\mathrm{i}\lambda x_1}c(x_1,x')\mathrm{d}x_1,
\]
$B_\delta(y_0)$ is a small neighborhood of $y_0$ and
\[
A=(\sqrt{\tau^2+k^2-\lambda^2})^{\frac{n-2}{2}}e^{-2\Im(s\Phi-\sqrt{\tau^2+k^2}\Phi)}|\chi_1|^2|\chi_2|^2|a|^2|b|^2.
\]
Notice that $|a|^2|b|^2=|a_0|^2|b_0|^2+\mathcal{O}((\tau^2+k^2-\lambda^2)^{-1/2})$, one can write
\[
\begin{split}
&\left(\sqrt{\tau^2+k^2-\lambda^2}\right)^{\frac{1}{2}}\int_{M}c v_1v_2v_3v_4\mathrm{d}V_g\\
=&\left(\sqrt{\tau^2+k^2-\lambda^2}\right)^{\frac{n-1}{2}}\int_{B_\delta(y_0)}\hat{c}(2\lambda,\cdot)e^{-2\sqrt{\tau^2+k^2}\Im(\Phi+\Psi)}A_0\mathrm{d}V_{g_0}+\mathcal{O}((k^2+\tau^2-\lambda^2)^{-1/4}e^{\sigma|\lambda|}),
\end{split}
\]
where
\[
A_0=e^{-2\Im(s\Phi-\sqrt{\tau^2+k^2}\Phi)}|\chi_1|^2|\chi_2|^2|a_0|^2|b_0|^2.
\]
Let us rearrange
\[
B_0(z;\tau,k,\lambda)=A_0e^{-2\sqrt{\tau^2+k^2}\Im(\Phi+\Psi)+2\sqrt{\tau^2+k^2-\lambda^2}\Im(\Phi+\Psi)},
\]
and write
\[
\int_{B_\delta(y_0)}\hat{c}(2\lambda,\cdot)e^{-2\sqrt{\tau^2+k^2}\Im(\Phi+\Psi)}A_0\mathrm{d}V_{g_0}=\int_{B_\delta(y_0)}\hat{c}(2\lambda,\cdot)e^{-2\sqrt{\tau^2+k^2-\lambda^2}\Im(\Phi+\Psi)}B_0\mathrm{d}V_{g_0}.
\]
Notice that
\[
\begin{split}
|\sqrt{\tau^2+k^2}-\sqrt{\tau^2+k^2-\lambda^2}|=&\left|\frac{\lambda^2}{\sqrt{\tau^2+k^2}+\sqrt{\tau^2+k^2-\lambda^2}}\right|\\
\leq&\left|\frac{\lambda^2}{\sqrt{\tau^2+k^2}}\right|\\
\leq & |\lambda|
\end{split}
\]
and
\[
\left|s-\sqrt{\tau^2+k^2}\right|\leq \sqrt{5}|\lambda|
\]
for $\tau^2+k^2\geq \lambda^2$, we have
\[
\|B_0\|_{C^1(M_0)}\leq e^{\sigma|\lambda|},\quad |B_0|\geq e^{-\sigma|\lambda|},
\]
for some constant $\sigma>0$.

 For simplicity, we denote $\varsigma=\sqrt{\tau^2+k^2-\lambda^2}$. We now have
\begin{equation}
\varsigma^{\frac{n-1}{2}}\int_{B_\delta(y_0)}\hat{c}(2\lambda,\cdot)e^{-2\varsigma\Im(\Phi+\Psi)}B_0\mathrm{d}V_{g_0}=\varsigma^{1/2}\int_{M}c v_1v_2v_3v_4\mathrm{d}V_g+\mathcal{O}(\varsigma^{-1/2}e^{\sigma|\lambda|}).
\end{equation}

Notice that, in a neighborhood of $y_0$,
\[
\Im(\Phi+\Psi)(z)=2z\cdot\nabla^2\Im(\Phi+\Psi)\vert_{z=y_0}z=z\cdot \mathcal{H}z+\widehat{\Theta}(z),
\]
with $|\widehat{\Theta}(z)|=\mathcal{O}(|z|^3)$. Here $\mathcal{H}=\nabla^2\Im(\Phi+\Psi)\vert_{z=y_0}$ is a positive definite matrix, because $\nabla^2\Im\Phi$ and $\nabla^2\Im\Psi$ are positive semidefinite, and positive definite in directions transversal to $\dot{\gamma}$ and $\dot{\eta}$ respectively.
 Then we write
\[
\begin{split}
&\varsigma^{\frac{n-1}{2}}\int_{B_\delta(y_0)}\hat{c}(2\lambda,\cdot)e^{-2\varsigma\Im(\Phi+\Psi)}B_0\mathrm{d}V_{g_0}\\
=&\varsigma^{\frac{n-1}{2}}\int_{B_\delta(y_0)}\hat{c}(2\lambda,\cdot)e^{-2\varsigma  x\cdot\mathcal{H}x}B_0\mathrm{d}V_{g_0}\\
&+\varsigma^{\frac{n-1}{2}}\int_{B_\delta(y_0)}\hat{c}(2\lambda,\cdot)e^{-2\varsigma x\cdot\mathcal{H}x}\left(e^{-2\varsigma\widehat{\Theta}(x)}-1\right)B_0\mathrm{d}V_{g_0}\\
=&:I_1+I_2.
\end{split}
\]

We will use the following lemma (cf. \cite[Lemma 6]{lassas2020uniqueness}):
\begin{lemma}
For $b\in C_c^1(\mathbb{R}^d)$, a compactly supported function, we have the estimate
\[
\left|b(z_0)-\left(\frac{\varsigma}{\pi}\right)^{d/2}\int_{\mathbb{R}^d}b(z)e^{-\varsigma|z-z_0|^2}\mathrm{d}z\right|\leq C\|b\|_{C^1(\mathbb{R}^d)}\varsigma^{-1/2}.
\]
\end{lemma}

Setting $d=n-1$ and using the above lemma, we have
\[
\begin{split}
I_1=&\varsigma^{\frac{n-1}{2}}\int_{B_\delta(y_0)}\hat{c}(2\lambda,\cdot)e^{-2\varsigma z\cdot\mathcal{H}z}B_0\mathrm{d}V_{g_0}\\
=&b_0(\tau,k,\lambda)\hat{c}(2\lambda,y_0)+\mathcal{O}(\varsigma^{-1/2}e^{\sigma|\lambda|}),
\end{split}
\]
where $b_0(\tau,k,\lambda)\neq 0$ is a constant depending on $\tau,k,\lambda$. Note that
\[
|b_0(\tau,k,\lambda)|=c_0B_0(y_0;\tau,k,\lambda)\geq c_0e^{-\sigma|\lambda|},
\]
for some constant $c_0$.

Next we estimate $I_2$. Recall that $\widehat{\Theta}(z)=\mathcal{O}(|z|^3)$, we have $\varsigma\widehat{\Theta}(\varsigma^{-1/2}z)=\tau^{-1/2}\mathcal{O}(|z|^3)$. Then we can deduce that
\[
\vert e^{-2\varsigma\widehat{\Theta}(\varsigma^{-1/2}z)}-1\vert\lesssim \varsigma^{-1/2}|z|^3e^{C\tau^{-1/2}|z|^3}.
\]
Therefore, one can estimate
\[
\begin{split}
|I_2|=&\left\vert\varsigma^{\frac{n-1}{2}}\int_{B_\delta(y_0)}\hat{c}(2\lambda,\cdot)e^{-2\varsigma z\cdot\mathcal{H}z}\left(e^{-2\varsigma\widehat{\Theta}(z)}-1\right)B_0\mathrm{d}V_{g_0}\right|\\
=&\left\vert\int_{B_\delta(y_0)}\hat{c}(2\lambda,\varsigma^{-1/2}z)e^{-2 z\cdot\mathcal{H}z}\left(e^{-2\varsigma\widehat{\Theta}(\varsigma^{-1/2}z)}-1\right)B_0(\varsigma^{-1/2}z)\mathrm{d}z\right\vert\\
=&\mathcal{O}(\varsigma^{-1/2}e^{\sigma|\lambda|}).
\end{split}
\]

Summarizing above estimates, we end up with
\[
|\hat{c}(2\lambda, y_0)|^2\lesssim \epsilon^2(k^2+\tau^2+\lambda^2)^{(3n+46)/8} e^{2D\tau}e^{\sigma|\lambda|}+(k^2+\tau^2-\lambda^2)^{-1/2}e^{\sigma|\lambda|}.
\]
We emphasize here that $y_0$ can be an arbitrary point in $M_0$. Therefore
\begin{equation}\label{est0}
|\hat{c}(2\lambda, x')|^2\leq C\epsilon^2(k^2+\tau^2+\lambda^2)^{(3n+46)/8} e^{2D\tau}e^{\sigma|\lambda|}+C(k^2+\tau^2-\lambda^2)^{-1/2}e^{\sigma|\lambda|}
\end{equation}
for any $x'\in M_0$, and the constant $C>0$ can be chosen to be uniform in $x'$ by a compactness argument.\\

Let $E=-\log \epsilon>0$.

\noindent\textbf{Case 1}: $k>E$. Notice that if $\|c\|_{H^1}\leq \mathcal{M}$, we have
\begin{equation}\label{log_est1}
\int_{M_0}\int_{|\lambda|\geq \frac{\log k}{2\sigma}} |\hat{c}(2\lambda,x')|^2\mathrm{d}\lambda\mathrm{d}V_{g_0}(x')\leq \frac{C \mathcal{M}}{(\log k)^2}.
\end{equation}
Take $\tau=1$ in \eqref{est0}, we have for $|\lambda|\leq\frac{\log k}{2\sigma}$
\begin{equation}\label{log_est2}
|\hat{c}(2\lambda, x')|^2\lesssim \epsilon^2k^{\frac{3n+46}{4}}e^{\frac{\log k}{2}}+k^{-1}e^{\frac{\log k}{2}}=k^{12+\frac{3n}{4}}\epsilon^2+k^{-1/2}.
\end{equation}
Combining \eqref{log_est1} and \eqref{log_est2}, we obtain
\[
\begin{split}
\|c\|_{L^2}^2&=\int_{M_0}\int |\hat{c}(2\lambda,x')|^2\mathrm{d}\lambda\mathrm{d}V_{g_0}(x')\\
&=\int_{M_0}\int_{|\lambda|\leq \frac{\log k}{2\sigma}} |\hat{c}(2\lambda,x')|^2\mathrm{d}\lambda\mathrm{d}V_{g_0}(x')+\int_{M_0}\int_{|\lambda|\geq \frac{\log k}{2\sigma}} |\hat{c}(2\lambda,x')|^2\mathrm{d}\lambda\mathrm{d}V_{g_0}(x')\\
&\lesssim k^{12+\frac{3n}{4}}(\log k)\epsilon^2+\frac{\log k}{\sqrt{k}}+\frac{1}{(\log k)^2}.
\end{split}
\]
Since $k>E$, $(\log k)^2>\frac{1}{2}((\log k)^2+(\log E)^2)$, we have
\[
\|c\|_{L^2}^2\lesssim k^{12+\frac{3n}{4}}(\log k)\epsilon^2+\frac{1}{(\log k)^2+(\log E)^2}.
\]

\noindent\textbf{Case 2}: $k\leq E$. Let $\rho=\frac{1}{4\sigma}\log(k^2+E^2)$.
Then
\[
\begin{split}
\|c\|_{L^2}^2=&\int_{M_0}\int |\hat{c}(2\lambda,x')|^2\mathrm{d}\lambda\mathrm{d}V_{g_0}(x')\\
=&\int_{M_0}\int_{0<|\lambda|<\rho} |\hat{c}(2\lambda,x')|^2\mathrm{d}\lambda\mathrm{d}V_{g_0}(x')+\int_{M_0}\int_{|\lambda|\geq \rho} |\hat{c}(2\lambda,x')|^2\mathrm{d}\lambda\mathrm{d}V_{g_0}(x').
\end{split}
\]
%Still one has the estimate
%\begin{equation}\label{est1}
%\int_{M_0}\int_{|\lambda|\leq \frac{\log k}{2\sigma}} |\hat{c}(2\lambda,x')|^2\mathrm{d}\lambda\mathrm{d}V_{g_0}\lesssim k^{12+\frac{3n}{4}}(\log k)\epsilon^2+\frac{\log k}{\sqrt{k}}.
%\end{equation}
For $0<|\lambda|<\rho$, take $\tau=\frac{E}{2D}$. We have
\[
\epsilon^2(k^2+\tau^2+\lambda^2)^{(3n+46)/8} e^{2D\tau}e^{\sigma|\lambda|}\leq \epsilon^2(k^2+\tau^2+\rho^2)^{(3n+46)/8}e^{2D\tau}e^{\sigma\rho}\lesssim \epsilon^2(k^2+E^2)^{12+\frac{3n}{4}}\epsilon^{-1},
\]
and
\[
(k^2+\tau^2-\lambda^2)^{-1/2}e^{\sigma|\lambda|}\lesssim \frac{1}{\sqrt{k^2+E^2}}(k^2+E^2)^{1/4}=(k^2+E^2)^{-1/4}.
\]
Therefore, by \eqref{est0} we have
\begin{equation}\label{singlef_est2}
|\hat{c}(2\lambda, y_0)|^2 \lesssim \epsilon E^{12+\frac{3n}{4}}+(k^2+E^2)^{-1/4},\quad\text{ for }0<|\lambda|<\rho.
\end{equation}
Since, using the fact $k\leq E$,
\[
\int_{0<|\lambda|<\rho} \mathrm{d}\lambda=\frac{1}{2\sigma}\log(k^2+E^2)\lesssim  \log E,
\]
 we have
\begin{equation}\label{est2}
\int_{0<|\lambda|<\rho} |\hat{c}(2\lambda,y_0)|^2\mathrm{d}\lambda\lesssim\epsilon E^{12+\frac{3n}{4}}\log E +\frac{\log E}{(E^2+k^2)^{1/4}}.
\end{equation}
Also noticing that
\begin{equation}\label{est3}
\int_{|\lambda|\geq \rho} |\hat{c}(2\lambda,y_0)|^2\mathrm{d}\lambda\leq\frac{\mathcal{M}}{1+\rho^2}\lesssim\frac{1}{(\log E)^2+(\log k)^2}.
\end{equation}
and combining \eqref{est2}, \eqref{est3}, we obtain the estimate
\[
\|c\|_{L^2}^2\lesssim E^{12+\frac{3n}{4}}(\log E)\epsilon+\frac{\log E}{(E^2+k^2)^{1/4}}+\frac{1}{(\log E)^2+(\log k)^2}.
\]
This completes the proof of Theorem \ref{maintheorem}.

\begin{remark}
Using \eqref{log_est2} and \eqref{singlef_est2}, we have that, for $|\lambda|\leq\frac{\log k}{2\sigma}$, a logarithmic stability estimate
\[
|\hat{c}(2\lambda,x')|^2\lesssim k^{12+\frac{3n}{4}}\epsilon^2+\epsilon E^{12+\frac{3n}{4}}+E^{-\frac{1}{2}}
\]
holds. In Euclidean space, a better Lipschitz stability estimate in a larger interval can be achieved \textnormal{(cf. \cite[Theorem 2.1]{lu2022increasing})}.
\end{remark}

\bibliographystyle{abbrv}

\end{document}